\documentclass[amstex,12pt]{article}
\usepackage{graphicx}

\usepackage{amsmath}
\usepackage{amssymb}
\usepackage{amscd}
\usepackage{bm}
\usepackage{latexsym}
\usepackage{makeidx}
\usepackage{tocloft}
\usepackage{fancyhdr}
\usepackage[all]{xy}

\usepackage{ifthen}

\newcommand{\lsection}[2][""]{%
    \ifthenelse{\equal{#1}{""}}{%
        \section{#2}
    }{%
        \renewcommand{\sectionmark}[1]{\markright{\thesection.\ \MakeUppercase{#1}}}
        \section{#2}
        \renewcommand{\sectionmark}[1]{\markright{\thesection.\ \MakeUppercase{##1}}}
    }
}

\newcommand{\lchapter}[2][""]{%
    \ifthenelse{\equal{#1}{""}}{%
        \chapter{#2}
    }{%
        \renewcommand{\chaptermark}[1]{\markboth{\MakeUppercase{\chaptername\ \thechapter.\ #1}}{}}
        \chapter{#2}
        \renewcommand{\chaptermark}[1]{\markboth{\MakeUppercase{\chaptername\ \thechapter.\ ##1}}{}}
    }
}

\def\a{\alpha}
\def\b{\beta}
\def\l{\lambda}
\def\iso{isomorphism}
\def\homo{homomorphism}
\def\ra{\rightarrow}
\def\vb{vector bundle}
\def\fcn{function}
\def\mfd{manifold}

\def\nbd{neighborhood}
\def\hra{\hookrightarrow}
\def\lra{\longrightarrow}

\def\U{{\cal U}}

\def\W{{\cal W}}
\def\D{{\cal D}}
\def\ssm{\hspace{-.5mm}\smallsetminus\hspace{-.5mm}}

\def\na{\nabla}
\def\o{\omega}

\def\t{\theta}

\def\vL{\varLambda}
\def\vP{\varPsi}
\def\vG{\varGamma}
\def\vv {\vskip.2cm}
\def\st{such that}
\def\op{\operatorname}

\def\C{{\mathbb C}}
\def\R{{\mathbb R}}
\def\Z{{\mathbb Z}}

\def\bo1{{\text{\bold 1}}}

\def\r{respectively}

\def\Randmarke#1{\vadjust{\vbox to 0pt
                {\vss \hbox to \hsize{\hskip\hsize
                                     \quad #1\hss}\vskip3.5pt}}}
\def\Randpfeilmarke#1{\Randmarke{$\scriptscriptstyle\Leftarrow$#1}}
\def\bi#1 { $^{\bf!}$ {\{\sl #1 \}} \Randpfeilmarke{\bf !!}}

\newtheorem{theorem}{Theorem}[section]

\newtheorem{corollary}[theorem]{Corollary}
\newtheorem{proposition}[theorem]{Proposition}
\newtheorem{exa}[theorem]{Example}
\newenvironment{example}{\begin{exa} \em}{\end{exa}}
\newtheorem{exas}[theorem]{Examples}

\newtheorem{prope}[theorem]{Property}

\newtheorem{defini}[theorem]{Definition}
\newenvironment{definition}{\begin{defini} \em}{\end{defini}}

\newtheorem{rema}[theorem]{Remark}
\newenvironment{remark}{\begin{rema} \em}{\end{rema}}

\newenvironment{equationth}{\stepcounter{theorem}\begin{equation}}{\end{equation}}

\newenvironment{proof}{{\noindent \sc Proof: } }{\mbox{ }\hfill$\Box$
                        \vspace{1.5ex} \par}

\def\hfleche#1#2{\smash{\mathop{\vbox{\hbox to 8.5mm{{#1}}}}\limits^{#2}}}

\title {Localized intersection of currents and\\
 the Lefschetz coincidence point theorem}
\author{Cinzia Bisi\thanks{Partially supported by PRIN2010-2011 Protocollo: 2010NNBZ78-012, by Firb2012 
 Codice: RBFR12W1AQ-001 and by GNSAGA-INDAM.}
, Filippo Bracci\thanks{Partially supported by the ERC grant ``HEVO - Holomorphic Evolution Equations'' n. 277691.}
, Takeshi Izawa and Tatsuo Suwa\thanks{Partially supported by the JSPS grant no. 24540060.}}

\date{}

\begin{document}


\pagestyle{plain}


\maketitle






\paragraph{Abstract\,:} We introduce the notion of a Thom class of a current and define the localized intersection of currents.
In particular we consider the situation where we have a $C^\infty$ map of \mfd s and study  localized
intersections of the source \mfd\ and currents on the target \mfd. We then obtain a residue theorem
on the source \mfd\ and give  explicit formulas for the residues in some cases.
These are applied to the problem of coincidence points of two maps.
We define the global and local coincidence homology classes and indices. 
A representation of the Thom class of the graph as a \v Cech-de~Rham cocycle immediately gives us an explicit expression of the  index at an isolated coincidence point, which in turn gives explicit coincidence classes in some non-isolated
components. Combining these, we have a general coincidence point theorem including the one by S. Lefschetz.
\vv


\vv



\section*{Introduction}

For two cycles in a \mfd, the localized intersection product of their classes   
is defined,   in the homology of their set theoretical intersection, via Alexander dualities and the cup product in the relative
cohomology.
Thus for a cycle $C$ in a \mfd\ $W$, the corresponding class   in the relative cohomology carries the local
information on $C$. For a sub\mfd\ $M$ of $W$, this is the Thom class of $M$, which may be identified with
the Thom class of the normal bundle of $M$ in $W$ by the tubular \nbd\ theorem.
In this paper we take up the localization problem of currents. We introduce the notion of a Thom class of a current
and define the localized intersection of two currents. In particular we consider the intersections of a fixed sub\mfd\ $M$
and currents on $W$, obtain a residue theorem on $M$ and give explicit expressions of the residues in some cases
(see Theorems \ref{res-m} and \ref{resnoniso} below for precise statements).
As an application we study the coincidence point problem for two maps.

The coincidence point formula discovered by S. Lefschetz (cf. \cite{Lef1}, \cite{Lef2}) is   formulated for a pair of continuous maps
between compact oriented  topological manifolds of the same dimension.
Using the above,  we define global and local homology classes of  
coincidence for a pair of $C^\infty$ maps $M\ra N$ and give a general coincidence point theorem, 
even in the case  the dimensions $m$ and $n$ of $M$ and $N$ are different ($m\ge n$) and
the coincidence points are non-isolated (Definitions \ref{gcoin} and \ref{lcoin}). In the case  $m=n$,
the use of Thom class in the \v Cech-de~Rham cohomology immediately gives us an explicit 
expression of the 
coincidence index at an isolated  coincidence point (Propositions \ref{propisolated}). This gives in turn an explicit coincidence homology class
at a certain non-isolated coincidence component  in the case $m>n$ (Proposition \ref{coinnonisol}).
We then have a general 
coincidence point formula, including the Lefschetz coincidence point formula.

The paper is organized as follows. In Section \ref{CdRTh}, we recall preliminary materials such as  Poincar\'e and Alexander dualities, global and localized intersection products,  Thom classes
in various settings and an explicit expression, in the \v Cech-de Rham cohomology, of the Thom class of an oriented \vb\ 
(Proposition \ref{Thomloc}). In Section \ref{secloc}, we consider the localization problem of currents, introduce the notion of 
a Thom class of a current and give some examples. We then consider localized intersections of currents.
In particular we study  intersections of a fixed sub\mfd\ $M$ in a \mfd\ $W$ and currents on $W$.
In fact we consider a more general situation where we have a map $F:M\ra W$ (Definition \ref{prodmap}).
 For a closed current $T$ on $W$, we have the intersection product $M\cdot_F T$ in the homology of $M$ and, if $T$ is localized at a compact set $\tilde S$ in $W$
and if $\vP_T$ is a Thom class of $T$ along $\tilde S$, we have the residue of $F^*\vP_T$  in the homology of $S=F^{-1}\tilde S$ as a localized intersection product. If $S$ has several connected components, we have a ``residue
theorem" (Theorem \ref{res-m}). We give an explicit formula for the residue at a non-isolated component of $S$ in the case it
is a sub\mfd\ of $M$ (Theorem \ref{resnoniso}). These are  conveniently used in Section \ref{seccoin}, where we study
the coincidence point problem for two maps. 

Let $M$ and $N$ be  \mfd s of dimensions $m$ and $n$ and let
$f, g:M\ra N$ be two maps.  We define the global coincidence class of the pair $(f,g)$ in the $(m-n)$-th homology of $M$
(Definition \ref{gcoin}). If $M$ is compact, we also define the local coincidence class of the pair  in the $(m-n)$-th homology of the  set of points in $M$ where $f$ and $g$ coincide (Definition \ref{lcoin}). We then
apply Theorem \ref{res-m} to get a general coincidence point theorem (Theorem \ref{gcth}).
In the case $m=n$ we have a formula for the coincidence index at an isolated coincidence point as the 
local mapping degree of $g-f$ (Proposition \ref{propisolated}).
This is in fact a classical result, however we give a short direct proof using the aforementioned expression of the Thom
class in the \v Cech-de Rham cohomology. This together with Theorem \ref{resnoniso} gives an explicit expression of the 
coincidence homology class at a non-isolated component (Proposition \ref{coinnonisol}). If $M$ and $N$ are compact
\mfd s of the same dimension, we have a general coincidence point formula (Theorem \ref{non-isolatecoin}), which
reduces to the Lefschetz coincidence point formula in the case the coincidence points are isolated (Corollary \ref{lefcpf}).

\section{Notation and conventions}\label{secnote}

For a topological space $X$, we  denote by $H_*(X,\C)$ and $H^*(X,\C)$ its homology and cohomology
 of finite singular chains with $\C$ coefficients. Also we denote by $\breve H_*(X,\C)$  the homology of locally finite 
 singular chains (Borel-Moore homology). For a $C^\infty$ \mfd, they can be computed using $C^\infty$
 simplicial chains. To be a little more precise, any $C^\infty$ \mfd\ $M$ admits a $C^\infty$ triangulation, which is essentially unique,  
and the groups $H_*(M,\C)$ and $H^*(M,\C)$ are naturally isomorphic with the ones defined by
  finite $C^\infty$ simplicial chains. Also the group $\breve H_*(M,\C)$  is naturally isomorphic with the homology of locally finite $C^\infty$ simplicial chains.

 In the sequel  a locally finite $C^\infty$ simplicial chain is simply called a chain,
unless otherwise  stated. Thus a chain $C$ is expressed as a locally finite sum $C=\sum a_is_i$ with  $a_i$  in $\C$ and $s_i$
 oriented $C^\infty$ simplices. We set $|C|=\bigcup s_i$ and call it the support of $C$. It is a closed set.
 For a cycle $C$, its class in the homology of the ambient space  is denoted by $[C]$, while the class in the homology
 of its support is simply denoted by $C$. 

For an open set $U$ in $M$, we
denote by $A^p(U)$ and $A^p_c(U)$, \r, the spaces of complex valued $C^{\infty}$ $p$-forms and $p$-forms with
compact support on $U$. The 
cohomology of the complex $(A^*(M),d)$ is the de~Rham cohomology
$H^*_{\rm dR}(M)$ and that of $(A^*_c(M),d)$ is the cohomology $H^*_c(M)$ with compact support.
A $C^\infty$ form will be  simply called a  form unless otherwise  stated.

\section{\v Cech-de~Rham cohomology and the Thom class}\label{CdRTh}

For the background on the \v Cech-de~Rham cohomology, we refer to \cite{BT}. The integration theory
on this cohomology  is developed in  \cite{Leh2}.
See  \cite{Su6} also for these materials and for the description of the Thom class in the framework of  relative \v Cech-de~Rham cohomology. The relation with the combinatorial viewpoint, as given in \cite{Br1},  is discussed in \cite{Su8}.

In this section we let $M$ denote a $C^\infty$ \mfd\ of dimension $m$. 

\subsection{Poincar\'e duality}\label{ssecPoin} 

We recall the Poincar\'e duality  and global intersection products of homology classes.

Suppose that $M$ is connected and oriented. Then  the pairing
\[
A^p(M)\times A^{m-p}_c(M)\lra\C\qquad\text{given by}\quad (\o,\varphi)\mapsto\int_M\o\wedge\varphi
\]
induces the Poincar\'e duality for a possibly non-compact \mfd\,:
\begin{equationth}\label{PoinDual}
P:H^p(M,\C)\simeq H^p_{\rm dR}(M)\stackrel{\sim}{\lra} H^{m-p}_c(M)^*\simeq\breve H_{m-p}(M,\C).
\end{equationth}
In the sequel we sometimes omit the coefficient $\C$ in homology and cohomology. In fact the Poincar\'e duality holds with $\Z$ coefficient. Note that $P$ is given by the left cap product with the fundamental class 
of $M$, the class of the sum of all $m$-simplices in $M$.
We also denote $P$ by $P_M$ if we wish  to make the \mfd\ $M$ under consideration explicit.

In the  \iso\ \eqref{PoinDual}, the class $[\o]$ of a closed $p$-form $\o$ corresponds 
to the functional
on $H^{m-p}_c(M)$ given by
\begin{equationth}\label{corrpoinform}
[\varphi]\mapsto\int_M\o\wedge\varphi,
\end{equationth}
or 
to the class $[C]$ of an $(m-p)$-cycle $C$ \st
\begin{equationth}\label{corrpoin}
\int_M\o\wedge\varphi=\int_C\varphi
\end{equationth}
for any closed form $\varphi$ in $A^{m-p}_c(M)$. We call $\o$ a {\em  de~Rham representative} of $C$.

For 
two classes $[C_1]\in \breve H_{q_1}(M)$ and $[C_2]\in \breve H_{q_2}(M)$, 
the {\em intersection product} $[C_1]\cdot [C_2]$ is 
defined by
\begin{equationth}\label{defint}
[C_1]\cdot [C_2]:=P(P^{-1}[C_1]\smallsmile P^{-1}[C_2])\qquad\text{in}\ \ \breve H_{q_1+q_2-m}(M),
\end{equationth}
where $\smallsmile$ denotes the cup product, which corresponds to the exterior product in the first \iso\ in \eqref{PoinDual}.

If $M$ is compact and connected, then $\breve H_0(M,\C)=H_0(M,\C)=\C$. Thus if $q_1+q_2=m$, $[C_1]\cdot [C_2]$ is a number given by
\[
[C_1]\cdot [C_2]=\int_M\o_1\wedge\o_2=\int_{C_1}\o_2=(-1)^{q_1q_2}\int_{C_2}\o_1,
\]
where $\o_1$ and $\o_2$ are de~Rham representatives of $C_1$ and $C_2$, \r.

\subsection{\v Cech-de~Rham cohomology}\label{CdR}

The \v Cech-de~Rham cohomology is defined for an arbitrary open covering of
$M$, however here we only consider  coverings  consisting of  two open sets.
Thus let $\U=\{U_0, U_1\}$ be an open covering of $M$. We set
$U_{01}=U_0\cap U_1$ and define the complex vector space $A^p(\U)$ as
\[
A^p(\U):=A^p(U_0)\oplus A^p(U_1)\oplus A^{p-1}(U_{01}).
\]
An element $\sigma$ in $A^p(\U)$ is given by a triple
$\sigma=(\sigma_0, \sigma_1, \sigma_{01})$ with $\sigma_i$ a $p$-form on $U_i$, $i=0,1$, and
$\sigma_{01}$  a $(p-1)$-form on $U_{01}$. We define an operator $D: A^p(\U) \to A^{p+1}(\U)$ by
\[
D\sigma:=(d\sigma_0, d\sigma_1, \sigma_1|_{U_{01}}-\sigma_0|_{U_{01}}-d\sigma_{01}).
\]
Then we see that  $D \circ D =0$ so that we have a complex
$(A^*(\U),D)$. The $p$-th {\em \v{C}ech-de~Rham cohomology} of $\U$, denoted by
$H_D^p(\U)$, is the $p$-th cohomology of this complex. It is also abbreviated as  \v CdR cohomology.
We denote the class of a cocycle $\sigma$ by $[\sigma]$.
It can be shown that the map $A^p(M) \to A^p(\U)$ given by
$\o \mapsto (\o|_{U_0}, \o|_{U_1}, 0)$
induces an isomorphism
\begin{equationth}\label{dRCdR}
\a:H^p_{\rm dR}(M)\stackrel\sim\lra  H^p_D(\U).
\end{equationth}
Note that $\a^{-1}$ assigns to the class of a \v CdR cocycle $(\sigma_0,\sigma_1,\sigma_{01})$ the class of the
closed form $\rho_0\sigma_0+\rho_1\sigma_1-d\rho_0\wedge\sigma_{01}$, where $\{\rho_0,\rho_1\}$ is a partition
of unity subordinate to $\U$.
\vv

Now we could define the cup product for \v CdR cochains and describe the Poincar\'e duality
in terms of the \v CdR cohomology as in \cite{Su6} in the case $M$ is compact.  However here we proceed as follows.
Let $M$  and $\U=\{U_0, U_1\}$ be as above. A  {\em system of honeycomb cells} adapted to $\U$ is a collection 
$\{R_0,R_1\}$ 
of two  submanifolds of $M$ of dimension $m$ with $C^\infty$ boundary 
having the following properties\,:
\smallskip

\noindent
(1) $R_i\subset U_i$ for $i=0,1$,
\smallskip

\noindent
(2) $\op{Int}\, R_0\cap \op{Int}\, R_1=\emptyset$ and
\smallskip

\noindent
(3)  $R_0 \cup R_1=M$,
\smallskip

\noindent
where $\op{Int}$ denotes the interior.
Suppose $M$ is oriented. Then $R_0$ and $R_1$ are naturally oriented. Let $R_{01}=R_0\cap R_1$ with  the orientation as the boundary of
$R_0$\,; $R_{01}=\partial R_0$, or equivalently,  the  orientation opposite to that of  the
boundary of $R_1$\,; $R_{01}=-\partial R_1$. We consider the pairing
\[
A^p(\U)\times A^{m-p}_c(M)\lra\C
\]
given by
\begin{equationth}\label{intCdR}
(\sigma,\varphi)\mapsto\int_{R_0}\sigma_0\wedge\varphi+\int_{R_1}\sigma_1\wedge\varphi+\int_{R_{01}}\sigma_{01}\wedge\varphi.
\end{equationth}
Then it induces the Poincar\'e duality \eqref{PoinDual} through the \iso\ $\a$ in \eqref{dRCdR}.

\subsection{Relative \v Cech-de~Rham cohomology and Alexander duality}\label{relCdR}

We introduce the relative \v Cech-de~Rham cohomology and describe the Alexander duality, which is used to define
localized intersection products.

Let  $S$ be a closed set in $M$. Letting
$U_0=M\ssm  S$ and $U_1$ a \nbd\ of $S$ in $M$, we
consider the covering $\U=\{U_0,U_1\}$ of $M$.
 If we set
\[
A^p(\U,U_0)=\{\,\sigma\in A^p(\U)\mid \sigma_0=0\,\},
\]
we see that $(A^*(\U,U_0),D)$ is a subcomplex of $(A^*(\U),D)$. We denote by
$H^p_D(\U,U_0)$ the $p$-th cohomology of this complex. From the short exact sequence
\[
0\lra A^*(\U,U_0)\stackrel{j^*}\lra A^*(\U)\stackrel{\iota^*}\lra A^*(U_0)\lra 0,
\]
where $j^*$ is the inclusion and $\iota^*$ is the \homo\ that assigns $\sigma_0$ to $\sigma=(\sigma_0,\sigma_1,\sigma_{01})$, we have the long exact sequence
\begin{equationth}\label{longexact}
\cdots\lra H^{p-1}_D(\U)\stackrel{\iota^*}\lra H^{p-1}_{\rm dR}(U_0)\stackrel{\delta^*}\lra H^p_D(\U,U_0)\stackrel{j^*}\lra H^p_D(\U)\stackrel{\iota^*}\lra H^p_{\rm dR}(U_0)\lra\cdots.
\end{equationth}
In the above, $\delta^*$ assigns the  class $[(0,0,-\t)]$ to the class of a closed $(p-1)$-form $\t$ on $U_0$.
Comparing with the long cohomology exact sequence for the pair $(M,M\ssm S)$, we have
 a natural \iso\ (see  \cite{Su8} for a precise proof)\,:
\[
H^p_D(\U,U_0)\simeq H^p(M,M\ssm S;\C).
\]

We  describe the Alexander duality in terms of the relative \v CdR cohomology  in the case $S$ is compact
and admits a regular \nbd\ (cf.  \cite{Su6}).  
Thus suppose  $M$ is oriented and let $\{R_0,R_1\}$ be a system of honeycomb cells adapted to $\U$.
We assume that $S$ is compact so that we may also assume that $R_1$ is compact. Consider the pairing
\begin{equationth}\label{relpair}
A^p(\U,U_0)\times A^{m-p}(U_1)\lra\C
\end{equationth}
given by
\[
(\sigma,\varphi)\mapsto\int_{R_1}\sigma_1\wedge\varphi+\int_{R_{01}}\sigma_{01}\wedge\varphi.
\]
Then it induces the Alexander \homo\,:
\begin{equationth}\label{alexhomo}
A:H^p(M,M\ssm S;\C)\simeq H^p_D(\U,U_0)\lra H^{m-p}_{\rm dR}(U_1)^*\simeq H_{m-p}(U_1,\C).
\end{equationth}
The \homo\ \eqref{alexhomo}
 depends on $U_1$ and is not an \iso\ in general. 
Here we consider the following hypothesis\,:
\begin{enumerate}
\item[(*)] there exists a triangulation of $M$ \st\ 
$S$ is (the polyhedron of) a subcomplex.
\end{enumerate}
In this case we also say that the triangulation is compatible with $S$.

\begin{remark} The triangulation above may simply be $C^0$ to have the Alexander duality 
\eqref{AlexDual} below. We mainly consider $C^\infty$ triangulations in order that, in an  expression as \eqref{corralex}, the integral in the right hand side  makes
sense. The hypothesis is certainly satisfied  if $S$ is the support of a chain.
It is also satisfied if $S$ is a $C^\infty$ 
sub\mfd\ of $M$.

Another  case of interest where the following arguments go through is that $S$ is a subanalytic set in a real analytic \mfd\ $M$, as in this case there is a subanalytic triangulation of $M$ compatible with $S$ (cf. \cite{Sh}) so that the integration makes sense, see Example \ref{exmult} below.
\end{remark}

Under the hypothesis (*), we may take as $U_1$ a regular \nbd\ of $S$ so that there is a
deformation retract  $U_1\ra S$.  We then have $H_{m-p}(U_1)\simeq H_{m-p}(S)$ and  
\eqref{relpair} induces the Alexander duality\,: 

\begin{equationth}\label{AlexDual}
A:H^p(M,M\ssm S;\C)\simeq H^p_D(\U,U_0)\stackrel{\sim}{\lra} H^{m-p}_{\rm dR}(U_1)^*\simeq H_{m-p}(S,\C).
\end{equationth}
Note that $A$ is also given by the left cap product with the fundamental class $M$. We also denote $A$ by $A_{M,S}$ if we wish  to make the pair $(M,S)$ explicit. 

In the  \iso\ \eqref{AlexDual}, the class $[\sigma]$ of a  $p$-cochain $\sigma$ corresponds to the functional
on $H^{m-p}_{\rm dR}(U_1)$ given by
\begin{equationth}\label{corralexform}
[\varphi_1]\mapsto\int_{R_1}\sigma_1\wedge\varphi_1+\int_{R_{01}}\sigma_{01}\wedge\varphi_1,
\end{equationth}
or to the class $[C]$ of an $(m-p)$-cycle $C$  in $S$ \st
\begin{equationth}\label{corralex}
\int_{R_1}\sigma_1\wedge\varphi_1+\int_{R_{01}}\sigma_{01}\wedge\varphi_1=\int_C\varphi_1
\end{equationth}
for any closed form $\varphi_1$ in $A^{m-p}(U_1)$. 

Denoting by $i:S\hra M$ the inclusion, we have the following commutative diagram\,:
\begin{equationth}\label{PoinAlex}
\SelectTips{cm}{}
\xymatrix{ H^p(M, M \ssm S)\ar[r]^-{\sim}_-A \ar[d]^{j^*}&H^{m-p}_{\rm dR}(U_1)^*\simeq  H_{m-p}(S)
\ar[d]^{i_*}\\
H^p(M) \ar[r]^-{\sim}_-P &H^{m-p}_c(M)^*\simeq \breve H_{m-p}(M).}
\end{equationth}

\begin{remark}\label{remimp}
1. 
In the above, the \homo\ $i_*:H_{m-p}(S)\ra \breve H_{m-p}(M)$ is the one naturally induced from $i$, while
$i_*:H^{m-p}_{\rm dR}(U_1)^*\ra H^{m-p}_c(M)^*$ 
is described as follows.
For any functional $F_1$ on $H^{m-p}_{\rm dR}(U_1)$, there is a corresponding cycle $C$ in $S$ and $i_*F_1$ is given by
\[
i_*F_1[\varphi]=\int_C\varphi\qquad\text{for}\ \ [\varphi]\in H^{m-p}_c(M).
\]
Alternatively, if $(0,\sigma_1,\sigma_{01})$ is a \v CdR representative of $A^{-1}F_1$, then 
\[
i_*F_1[\varphi]=\int_{R_1}\sigma_1\wedge\varphi+\int_{R_{01}}\sigma_{01}\wedge\varphi.
\]
\smallskip

\noindent
2. For a closed set $S$ (which may not be compact) in $M$ satisfying (*), we may define the Alexander \iso
\[
A:H^p(M,M\ssm S)\stackrel{\sim}{\lra} \breve H_{m-p}(S)
\]
via combinatorial topology (cf. \cite{Br1}). 
\end{remark}

Let $S_1$ and $S_2$ be compact sets in $M$ satisfying (*) and set $S=S_1\cap S_2$. Let $A_1$, $A_2$ and $A$ denote the Alexander \iso s for $(M,S_1)$, $(M,S_2)$ and $(M,S)$, \r.
For two classes $c_1\in  H_{q_1}(S_1)$ and $c_2\in  H_{q_2}(S_2)$, the {\em localized intersection
product} $(c_1\cdot c_2)_S$ is defined by
\begin{equationth}\label{locint}
(c_1\cdot c_2)_S:=A(A_1^{-1}c_1\smallsmile A_2^{-1}c_2)\qquad\text{in}\ \  H_{q_1+q_2-m}(S),
\end{equationth}
where $\smallsmile$ denotes the cup product
\[
H^{m-q_1}(M,M\ssm S_1)\times H^{m-q_2}(M,M\ssm S_2)\stackrel\smallsmile\lra H^{2m-q_1-q_2}(M,M\ssm S).
\]
Letting $i_1:S_1\hra M$, $i_2:S_2\hra M$ and $i:S\hra M$ be the inclusions, from 
\eqref{PoinAlex} we see that the definitions \eqref{defint} and \eqref{locint} are consistent in the sense that
\[
i_*(c_1\cdot c_2)_S=(i_1)_*c_1\cdot (i_2)_*c_2.
\]

\subsection{Thom class}\label{Thom}

We list \cite{MS} as a general reference for the Thom \iso\ and the Thom class of a  real \vb. In general
they are defined in cohomology with $\Z_2$ coefficients, while
for an oriented \vb, they can be defined in cohomology with $\Z$ coefficients.
They can also be described in terms of differential forms,
in which case the cohomology involved is with $\C$ coefficients. This is done
in \cite{BT} using cohomology with compact support in the vertical direction.
Here we use \v Cech-de~Rham cohomology instead as in \cite{Su6}. This way we can express relevant local informations
more explicitly.

In this subsection, we sometimes omit the coefficient $\C$ in homology and cohomology. 
In fact the \iso s we consider below can be defined from combinatorial viewpoint in homology and cohomology with  $\Z$ coefficient (cf. \cite{Br1}, also \cite{Su8}).

\paragraph{(a) Thom class of an oriented real \vb\,:} Let 
$\pi:E\to M$ be an oriented real \vb\ of rank $k$. We identify $M$ with the image of the zero section. Then we have the Thom \iso\
\[
T_E:H^p(M,\C)\overset{\sim}{\lra}H^{p+k}(E,E\ssm M;\C),
\]
 whose inverse is given by the integration   along the fiber   of $\pi$ (see \cite[Ch.II, 5]{Su6}). 
 
 The {\em Thom class 
 $\vP_E$} of $E$, which is in $H^{k}(E,E\ssm M)$, is the image of the constant function $1$ in $H^0(M,\C)$ by $T_E$.
 Note that $T_E$ is given by the cup product with $\vP_E$. Let $W_0= E\ssm M$ and $W_1$ a \nbd\ of $M$ in $E$ and consider the covering $\W=\{W_0,W_1\}$ of  $E$. We refer to \cite[Ch.II, Proposition 5.7]{Su6} for an explicit
 expression of a \v CdR cocycle representing $\vP_E$ 
in the \iso\ $H^k(E,E\ssm M)\simeq H^k_D(\W,W_0)$.
In particular suppose $E$ is trivial on an open set $U$ of $M$. Then, setting $A^k(\W,W_0)|_U=A^k(\W',W_0')$ with
$W_i'=W_i\cap\pi^{-1}(U)$, we have (cf. \cite[Ch.III, Lemma 1.4]{Su6})\,:

\begin{proposition}\label{Thomloc}
Suppose $E$ is trivial on an open set $U$ of $M$\,; $E|_U\simeq\R^k\times U$ and let $\rho:E|_U\ra\R^k$ denote the
projection on to the fiber direction. Then $\vP_{E|_U}$ is represented by a cocycle in $A^k(\W,W_0)|_U$ of the form
\[
(0,0,-\rho^*\psi_k),
\]
where $\psi_k$ is an angular form on $\R^k\ssm\{0\}$, i.e., a closed $(k-1)$-form with $\int_{S^{k-1}}\psi_k=1$.
\end{proposition}

Suppose $M$ is compact and oriented. We orient the total space $E$ so that, if $\xi=(\xi_1,\dots,\xi_k)$ is a positive fiber coordinate system of $E$ and if $x=(x_1,\dots,x_m)$ is a positive coordinate system on $M$, then $(\xi,x)$ is a positive coordinate system on $E$.
We then have the commutative diagram\,:

\begin{equationth}\label{ThAP}
\SelectTips{cm}{}
\xymatrix{H^p(M) \ar[r]_-{T_E}^-{\sim} \ar[d]^{P}_{\wr}& H^{p+k}(E,E\ssm M)
\ar[ld]^{A}_{\rotatebox{30}{{\hspace{-3.5mm}}$\sim$}}  \\
H_{m-p}(M).&{}}
\end{equationth}

\paragraph{(b) Thom class of a sub\mfd\,:} 
Let $W$ be an oriented $C^\infty$ \mfd\ of dimension $m'$ and $M$ a compact and 
oriented sub\mfd\ of $W$ of dimension $m$. Set $k=m'-m$. 
In view of (\ref{ThAP}), we define the Thom \iso\ $T_M:H^p(M)\overset{\sim}\ra H^{p+k}(W,W\ssm M)$ 
so that the following  diagram becomes commutative (cf. \cite{Br1})\,:
\[
\SelectTips{cm}{}
\xymatrix{H^p(M) \ar[r]_-{T_M}^-{\sim} \ar[d]^{P}_{\wr}& H^{p+k}(W,W\ssm M)
\ar[ld]^{A}_{\rotatebox{30}{{\hspace{-3.5mm}}$\sim$}}  \\
H_{m-p}(M).&{}}
\]
Then  the Thom class $\vP_M$ of $M$
is defined by
\begin{equationth} \label{ThomClass}
\vP_M=:T_M(1)=A^{-1}(M)\qquad\text{in}\ \ H^{k}(W,W\ssm M).
\end{equationth}

Let $N_M\to M$ denote the normal bundle of $M$ in $W$. Suppose the orientation of $M$ is compatible with that of $W$
in the sense that $N_M$ is orientable.
We orient $N_M$ as follows. Namely, if $(x_1,\dots,x_k,\dots,x_{m'})$ is a positive coordinate system on $W$ such that $M$ is given by $x_1=\cdots=x_k=0$ and that $(x_{k+1},\dots,x_{m'})$ is a positive coordinate system of $M$, then the vectors $(\frac{\partial}{\partial x_1},\dots,\frac{\partial}{\partial x_k})$ determine a positive frame of $N_M$. By the tubular \nbd\ theorem, there is a \nbd\ $W_1$ of $M$ in $W$ and an orientation
preserving diffeomorphism of $W_1$ onto a \nbd\ of the zero section of $N_M$, which is identified with $M$.
By excision we have
\[
H^k(W,W\ssm M)\simeq
H^k(W_1,W_1\ssm M)\simeq H^k(N_M,N_M\ssm M)
\]
and, in the above \iso s, the Thom class $\vP_M$ of $M$ corresponds to the Thom class $\vP_{N_M}$ of the \vb\ $N_M$.

\begin{remark}\label{rempseudo} More generally, for a pseudo-\mfd\ $M$ in $W$, we may define the Poincar\'e \homo\
$P:H^p(M)\lra\breve H_{m-p}(M)$.
Thus we have the Thom \homo\ and the Thom class $\vP_M$ of $M$ (cf.  \cite{Br1}).
\end{remark}
\paragraph{(c) Thom class of a cycle\,:} Let $C$ be a finite $(m'-p)$-cycle in $W$ and $\tilde S$ its support.
We may define the Thom class $\vP_C$ of $C$ by
\[
\vP_C:=A^{-1}(C),
\]
where $A$ is the Alexander \iso\ \eqref{AlexDual} for the pair $(W,\tilde S)$.

\section{Localized intersection of currents}\label{secloc}

\subsection{Thom class of a current}\label{thomclassofcurrents}

Let $W$ be an oriented  $C^\infty$ manifold of dimension $m'$. Recall that a  {\em $p$-current} $T$ on $W$ is a continuous linear functional
on the space $A_c^{m'-p}(W)$. We use the notation
\[
T(\varphi)=\langle T,\varphi\rangle,\qquad\varphi\in A_c^{m'-p}(W).
\]
Let $\D^p(W)$ denote the space of $p$-currents on $W$. The differential $d:\D^p(W)\to \D^{p+1}(W)$ is defined by
\[
\langle dT,\varphi\rangle=(-1)^{p+1}\langle T,d\varphi\rangle,\qquad\varphi\in A_c^{m'-p-1}(W).
\]
Then  $(\D^*(W),d)$ forms a complex, whose $p$-th cohomology is denoted by $H^p(\D^*(W))$.
For a closed $p$-current $T$, we denote by $[T]$ its cohomology class. 

A form $\o$ in $A^p(W)$ may be naturally thought of as a $p$-current $T_\o$ by
\[
\langle T_\o,\varphi\rangle=\int_W\o\wedge\varphi,\qquad\varphi\in A_c^{m'-p}(W).
\]
If $\o$ is closed, then  $T_\o$ is closed and the assignment $\o\mapsto T_\o$ induces an \iso
\begin{equationth}\label{dRcurr}
\b:H^p_{\rm dR}(W)\stackrel\sim\lra H^p(\D^*(W)).
\end{equationth}

A \v Cech-de~Rham cochain $\sigma$ on a covering $\W$ of $W$ may be also thought of as a current $T_\sigma$ via integration given as \eqref{intCdR}. If $D\sigma=0$, then  $T_\sigma$ is closed and the assignment $\sigma\mapsto T_\sigma$ induces
the \iso\ $\b\circ\a^{-1}:H^p_D(\W)\stackrel\sim\ra H^p(\D^*(W))$.
 
Also an $(m'-p)$-chain $C$ may be thought of as a $p$-current
$T_C$ by 
\[
\langle T_C,\varphi\rangle=\int_C\varphi,\qquad\varphi\in A_c^{m'-p}(W).
\]
If $C$ is a cycle, $T_C$ is closed and the assignment $C\mapsto T_C$ induces an \iso
\[
\gamma:\breve H_{m-p}(W)\stackrel\sim\lra H^p(\D^*(W)).
\]
By \eqref{corrpoin}, we have $\gamma\circ P=\b$. Thus for any closed current $T$ on $W$, there exist
a closed $p$-form $\o$, a \v CdR cocycle $\sigma$ and an $(m'-p)$-cycle $C$ \st
\[
[T]=[T_\o]=[T_\sigma]=[T_C].
\]
We call $\o$, $\sigma$ and $C$, \r,  de~Rham, \v CdR and cycle representatives of $T$.

If $U$ is an open set of $W$, there is a natural inclusion $A^*_c(U)\hra A^*_c(W)$, given by extension by zero, so that we may consider the restriction
$T|_U$ to $U$ of a current $T$ on $W$. The {\em support} $\op{supp}(T)$ of $T$ is the smallest closed subset of $W$ \st\
$T|_{W\ssm \op{supp}(T)}=0$.
\vv

Now we consider the localization problem of currents. 
Thus let   $\tilde S$ be  a closed set in $W$. Let $W_0=W\ssm \tilde S$ and  $W_1$  a 
\nbd\ of $\tilde S$ in $W$ and consider the covering
$\W=\{W_0, W_1\}$ of $W$. 
We have the commutative diagram with exact row (cf. \eqref{longexact})\,:

\begin{equationth}\label{funddiag}
\SelectTips{cm}{}
\xymatrix{H^p_D(\W, W_0)\ar[r]^-{j^*} &H^p_D(\W)\ar[r]^-{\iota^*}\ar[d]^{\b\circ\a^{-1}}_{\wr} &H^p_{\rm dR}(W_0)\ar[d]^{\b}_{\wr}\\
 {} & H^p(\D^{*} (W))\ar[r]^{\iota^*}&H^p(\D^{*} (W_0)).}
\end{equationth}

Suppose $T$ is a closed $p$-current on $W$ \st\ $\iota^*[T]=0$, i.e., $[T|_{W_0}]=0$. Then there is a class $\vP_T$ in $H^p_D(\W, W_0)$ \st\ 
$[T]=j^\ast\vP_T$. We then say that $T$ is {\em localized at $\tilde S$} and  call $\vP_T$ a {\em  Thom class of $T$ along $\tilde S$}. Here some comments are in order\,:
\begin{enumerate}
\item[(1)]  Any closed current $T$ is localized at $\op{supp}(T)$ in the above sense. It is also localized
at the support of a cycle representive of  $T$.
Thus the set $\tilde S$ as above 
may be different from $\op{supp}(T)$,
see Example \ref{smoothcase} below.
\item[(2)]  The class $\vP_T$ is not uniquely determined, as $j^*$ is not injective in general, however in some cases, there is a natural choice of $\vP_T$, see Examples \ref{Thom-submnf} and \ref{Cherncurrent} below.
\end{enumerate}

In the above situation, suppose $\tilde S$ is a compact set satisfying (*).
Let $(0,\psi_1,\psi_{01})$ be a \v CdR representative of $\vP_T$ and $\{\tilde R_0,\tilde R_1\}$ a system of honeycomb cells adapted to $\W$. Then, from the commutativity of the
diagram obtained by replacing $M$ and $\U$ by $W$ and $\W$ in \eqref{PoinAlex} (see also Remark \ref{remimp}), for any closed  form $\varphi$ in $A^{m'-p}_c(W)$ we have\,:
\begin{equationth}\label{localization}
\langle T, \varphi \rangle 
=\int_{\tilde R_1}\psi_1\wedge\varphi +\int_{\tilde R_{01}} \psi_{01}\wedge\varphi.
\end{equationth}
Thus the value of $T$ is ``concentrated" near $\tilde S$ and  is explicitly given by the above. Also note that the right
hand side does not depend on the choice of $\vP_T$.

\begin{example}\label{Thom-submnf}
Let $C$ be a finite $(m'-p)$-cycle  in  $W$. Then $T_C$ is localized at
$\tilde S=|C|$ and $\vP_C$ is a natural choice for $\vP_{T_C}$.  In particular, if $C=M$ is a compact oriented submanifold of codimension $p$, then the Thom class $\vP_M$ of $M$ is a
natural choice for $\vP_{T_M}$.
\end{example}

We show that, starting from a closed $p$-form $\o$ \st\ $[T_\o]=[T]$, there is a natural way of constructing
such a class. Thus let $\tilde S$ be a closed set of $W$ and let $\W=\{W_0,W_1\}$ be as before.

\begin{proposition}\label{m1}
Let $T$ be a closed $p$-current on $W$ such that $[T|_{W_0}]=0$ and $\o$ a de~Rham representative of $T$.
Then there exists a Thom class $\vP_T$
which is represented by a \v CdR cocycle of the form $(0,\o,-\psi)$ with $\psi$  a $(p-1)$-form on
$W_{01}$ satisfying $\o=-d\psi$ on
$W_{01}$.
\end{proposition}

\begin{proof} Just to make sure we denote the restrictions of forms explicitly.
From the assumption, there exists a $(p-1)$-form $\psi$ on $W_0$
such that $d\psi=-\o|_{W_0}$. Hence  the cocycle $(\o|_{W_0}, \o|_{W_1}, 0)$  is cohomologous as a \v CdR cocycle to
$(0, \o|_{W_1}, -\psi|_{W_{01}})$ since
\[
(0, \o|_{W_1}, -\psi|_{W_{01}})-(\o|_{W_0}, \o|_{W_1}, 0)=(d\psi, 0, -\psi|_{W_{01}})=D(\psi, 0, 0).
\]
Thus the class $\vP_{T}=[(0, \o|_{W_1}, -\psi|_{W_{01}})]$ satisfies $[T]=j^*\vP_T$.
\end{proof}

In this case, if $\tilde S$ is a compact set satisfying (*), \eqref{localization} is written as
\[
\int_W\o\wedge\varphi=\int_{\tilde R_1} \o\wedge\varphi -\int_{\tilde R_{01}} \psi\wedge\varphi.
\]
Thus the value of the integral away from $\tilde R_1$ is cut off  and is compensated by an integral on $\tilde R_{01}$.

\begin{example}\label{smoothcase} Let $C$ be an $(m'-p)$-cycle  in  $W$ and
 $\o$  a de~Rham representative of $C$.
Let $\tilde S$ be the support of $C$ and set $W_0=W\ssm \tilde S$. Then $T_\o$ is localized at
$\tilde S$ as $[T_\o]=[T_C]$, although we do not have any precise information
about $\op{supp}(T_{\o})$.
Its   Thom class  $\vP_{T_{\o}}$ 
along $\tilde S$ is represented  by a \v CdR cocycle of the form
$(0, \o, -\psi)$.
\end{example}

\begin{remark} Let $C$, $\o$ and  $\tilde S$ be as in Example \ref{smoothcase}.
Then there is a $(p-1)$-current $R$ \st
\[
T_C-T_\o=dR.
\]
We may think of $R$ as the current defined by a
$(p-1)$-form $\psi$ on $W\ssm \tilde S$ that can be extended as a locally integrable $L^1$ form on $W$
and with $d\psi=-\o$ on $W\ssm \tilde S$. The equation above becomes then
\[
dT_\psi-T_{d\psi}=T_C,
\]
which is a residue formula (cf. \cite[Ch.3,1]{GH}), and the identitiy
\[
D(\psi, 0, 0)+(\o, \o, 0)=(0, \o, -\psi)
\]
may be thought of as the corresponding expression in terms of \v CdR cochains.
\end{remark}

\begin{example}\label{Cherncurrent}
Let $\pi:E\ra W$ be a $C^\infty$ complex vector bundle of rank $r$ and $\na$ a connection for $E$. For $q=0,\dots,r$, 
we have the $q$-th Chern form $c_q(\na)$, which is a closed $2q$-form defining the $q$-th Chern class $c_q(E)$  
 in 
$H^{2q}_{\rm dR}(W)$. We call $T_{c_q(\na)}$ the $q$-th {\em Chern current associated with $\na$}.
Suppose  $E$ admits  $\ell$ sections $\bm{s}=(s_1,\ldots,s_\ell)$
that are linearly independent on the complement of a closed set $\tilde S\subset W$. Then we see that $T_{c_q(\na)}$
is localized at $\tilde S$ and there is a natural
way of choosing a Thom class along $\tilde S$ for $q=r-\ell+1,\dots,r$.

For this, we take an $\bm{s}$-trivial connection $\na_0$ for $E$ on $W_0=W\ssm \tilde S$, i.e., a connection satisfying 
$\na_0 s_i=0$ for $i=1,\ldots, \ell$.  
Denoting by $c_q(\na_0, \na)$ the  Bott difference form (cf. \cite{Bo1}, \cite{Su6}), we have 
$c_q(\na)|_{W_0}-c_q(\na_0)=d\, c_q(\na_0, \na)$.
Since the connection $\na_0$ is 
$\bm{s}$-trivial, it follows that $c_q(\na_0)=0$ so that $c_q(\na)|_{W_0}$ is exact for $q=r-\ell+1,\ldots, r$.
Hence the Chern current $T_{c_q(\na)}$ localizes at $\tilde S$. 
As its Thom class along $\tilde S$, we may take the class  $c_q(E,\bm{s})$ in 
$H^{2q}(W,W\ssm \tilde S)\simeq H^{2q}_D(\W,W_0)$ represented by the cocycle (cf. Proposition \ref{m1})\,:
\[
(0,c_q(\na)|_{W_1}, c_q(\na_0, \na)).
\]
 This class does not depend on the choice of $\na$
or $\na_0$ (cf. \cite[Ch.III, Lemma 3.1]{Su6}) and is a natural choice of Thom class for  $T_{c_q(\na)}$. 
It is the {\em localization of $c_q(E)$ at $\tilde S$ by $\bm{s}$}.

The Thom class of a complex \vb\ as a real oriented bundle may be expressed in this manner 
(cf. \cite[Ch.III, Theorem 4.4]{Su6}).
\end{example}

\subsection{Localized intersection of currents} \label{sec:localized}

Let $W$ be an oriented  $C^\infty$ manifold of dimension $m'$ as before.
For  closed currents $T_1$ and  $T_2$, 
it is possible to define the intersection product $T_1\cdot T_2$ in the homology of $W$ using the \iso\ $\b$
and the Poincar\'e duality (cf. \eqref{dRcurr}, \eqref{defint}). Also if $T_1$ and $T_2$ are localized at compact sets 
$\tilde S_1$ and $\tilde S_2$ satisfying (*),
we may define the localized intersection $(T_1\cdot T_2)_{\tilde S}$ in the homology of $\tilde S=\tilde S_1\cap \tilde S_2$ using Thom
classes $\vP_{T_1}$ and $\vP_{T_2}$ and the Alexander duality (cf. \eqref{locint}).

Here we consider the case $T_1=T_M$  with $M$ a compact oriented sub\mfd\ of 
dimension $m$ in $W$ and  obtain a residue theorem on $M$.
In the sequel we take the Thom class $\vP_M$ of $M$ (cf. Subsection \ref{Thom} (b))   as $\vP_{T_M}$ and set $k=m'-m$. Recall that by the Alexander \iso
\[
A_{W,M}:H^k(W,W\ssm M)\stackrel\sim\lra  H_m(M),
\]
the class $\vP_M$ corresponds to the fundamental class $M$. Let $i: M \hra W$ denote the  inclusion.

\paragraph{First localization\,:} Let $c$ be a class in $\breve H_{m'-p}(W)$. Recall that we have the intersection product 
$(M\cdot c)_M$ localized at $M$ (cf. \eqref{locint}), which is a class in $H_{m-p}(M)$ defined as 
$A_{W,M}(\vP_M\smallsmile P_W^{-1}c)$. We denote it by $M\cdot c$\,:
\[
M\cdot c:=(M\cdot c)_M.
\]
It is sent to $[M]\cdot c$ by $i_*: H_{m-p}(M)\ra\breve H_{m-p}(W)$.

\paragraph{Second localization\,:} Let $\tilde S$ be a compact set  of $W$ satisfying (*).  We have the Alexander
\iso
\[
A_{W,\tilde S}:H^p(W,W\ssm \tilde S)\stackrel\sim\lra  H_{m'-p}(\tilde S).
\]
We set $S=\tilde S\cap M$ and suppose it also satisfies (*).
For a class $c$ in $H_{m'-p}(\tilde S)$, we have the 
class $(M\cdot c)_S$ in $H_{m-p} (S)$ (cf. \eqref{locint}). 

\begin{proposition}\label{diagfund} The following diagrams are  commutative\,:
\[
\SelectTips{cm}{}
\xymatrix{ H^p(W) \ar[d]_{i^*}\ar[r]^-\sim_-{P_W}&\breve H_{m'-p}(W)\ar[d]^{M\cdot\ }\\
 H^p(M)\ar[r]^-\sim_-{P_M}&   H_{m-p}(M),}\hspace{1.7cm}
 \xymatrix{ H^p(W, W\ssm \tilde S) \ar[d]_{i^*}\ar[r]^-\sim_-{A_{W,\tilde S}}&\breve H_{m'-p}(\tilde S)\ar[d]^{(M\cdot\ )_S}\\
 H^p(M, M \ssm S)\ar[r]^-\sim_-{A_{M,S}}&   H_{m-p}(S).}
\]
\end{proposition}

\begin{proof} We prove the commutativity of the second diagram,  the proof for the first one  being  similar.
We  have the cup product followed by the Alexander isomorphism\,:
\[
H^k(W, W\ssm  M)\times H^p(W, W\ssm  \tilde S) \stackrel\smallsmile\lra
  H^{k+p}(W, W\ssm  S)\stackrel{A_{W, S}}\lra H_{m-p}(S).
\]
Nothing that the Alexander \iso\ is given by the left cap product with the fundamental class and using properties of cap and cup products, we have,
for a class $u$ in $H^p(W, W\ssm  \tilde S)$,
\[
  A_{M, S}(i^*u) =A_{W, S}(\vP_M\smallsmile u)=(M\cdot A_{W,\tilde S}u)_S.
\]
\end{proof}

In view of the above,  we define intersection products in a more general situation where $M$ is not necessarily a sub\mfd\ of $W$\,:

\begin{definition}\label{prodmap} Let $W$ and $M$ be oriented $C^\infty$ \mfd s of dimensions $m'$ and $m$, \r, and $F:M\ra W$  a $C^\infty$ map.  We define the intersection product $M\cdot_F\ $ so that the first diagram
below is commutative. Also, for a compact set $\tilde S$ satisfying (*) in $W$, we set $S=F^{-1}(\tilde S)$ and suppose 
$S$ is compact and satisfy (*).  We then define the localized intersection product $(M\cdot_F\ )_S$ 
so that the second diagram is commutative\,:
\[
\SelectTips{cm}{}
\xymatrix{ H^p(W) \ar[d]_{F^*}\ar[r]^-\sim_-{P_W}&\breve H_{m'-p}(W)\ar[d]^{M\cdot_F\ }\\
 H^p(M)\ar[r]^-\sim_-{P_M}& \breve  H_{m-p}(M),}\hspace{1.7cm}
 \xymatrix{ H^p(W, W\ssm \tilde S) \ar[d]_{F^*}\ar[r]^-\sim_-{A_{W,\tilde S}}&\breve H_{m'-p}(\tilde S)\ar[d]^{(M\cdot_F\ )_S}\\
 H^p(M, M \ssm S)\ar[r]^-\sim_-{A_{M,S}}&   H_{m-p}(S).}
\]
\end{definition}

\begin{remark} 1.  Let $M$ be  a  sub\mfd\ of $W$ and $i:M\hra W$  the inclusion.
If $M$ is compact,  $M\cdot_i \ $ is the product
 $M\cdot\ $ defined before. We may also define the  product $M\cdot\ $ as $M\cdot_i \ $ in the case 
 $M$ is not compact.
\smallskip

\noindent
2. The products as above are defined in the algebraic category in \cite{Fu}.
\end{remark}

For a closed $p$-current $T$ on $W$, we define
\[
M\cdot_FT:=M\cdot_FP\b^{-1}[T].
\]
Suppose $T$ is localized at $\tilde S$. Then taking a Thom class $\vP_{T}$   of $T$ along $\tilde S$, we define 
the {\em residue}  of $\vP_{T}$ on $M$ at $S$ by
\[
\op{Res}(F^*\vP_T,S):=(M\cdot_FA(\vP_T))_S.
\]

Suppose $S$ has a finite number of connected components $(S_\l)_\l$. Then we have a decomposition
$H_{m-p} (S)=\bigoplus_\l H_{m-p} (S_\l)$ and accordingly $\op{Res}(F^*\vP_T,S)$ determines
a class in $H_{m-p} (S_\l)$, which  is denoted by $\op{Res}(F^*\vP_T,S_\l)$.
We can state the following  general residue theorem, which follows from the commutativity of the diagram \eqref{PoinAlex}\,:

\begin{theorem}\label{res-m} Let $W$ and $M$ be oriented $C^\infty$ \mfd s of dimensions $m'$ and $m$, \r, and $F:M\ra W$  a 
$C^\infty$ map.  Let $T$ be a closed $p$-current on $W$ \st\ $[T|_{W\ssm \tilde S}]=0$ for some  compact subset 
$\tilde S$ satisfying (*) in $W$.
Suppose that  $S=F^{-1}\tilde S$ is compact, satisfies (*) and  has a finite number of connected components $(S_\l)_\l$. Then

\begin{enumerate}
\item[{\rm (1)}] For each $\l$ we have a class $\op{Res}(F^*\vP_T,\, S_\l)$ in $H_{m-p}(S_\l)$.
\item[{\rm (2)}] We have the ``residue formula"\,:
\[
M\cdot_F T=\sum_\l (i_\l)_*\op{Res}(F^*\vP_T,\, S_\l)\qquad\text{in}\ \ \breve H_{m-p}(M),
\]
where $i_\l:S_\l\hra M$ denotes the inclusion.
\end{enumerate}
\end{theorem}

We may express $M\cdot_F T$ and $\op{Res}(F^*\vP_T,\, S_\l)$ as follows.
Let $T$ be a closed $p$-current on $W$ and $\o$ a de~Rham representative of $T$. From \eqref{corrpoinform}
and \eqref{corrpoin}, we have\,:

\begin{proposition}\label{prodindR} {\rm (1)} The intersection product $M\cdot_F T$ in $\breve H_{m-p}(M)$ is represented
by a cycle $C$ \st\
\[
\int_MF^*\o\wedge\varphi=\int_C\varphi
\]
for any closed form $\varphi$ in $A_c^{m-p}(M)$.
\vv

\noindent
{\rm (2)} In the \iso\ 
$\breve H_{m-p}(M)\simeq H^{m-p}_c(M)^*$,  $M\cdot_F T$ corresponds to the functional on $H^p_c(M)$ that
assigns to $[\varphi]$ the left hand side above.
\vv

\noindent
{\rm (3)}
In particular, if $p=m$ and if $M$ is compact, $M\cdot_F T$ is a number given by
\[
M\cdot_F T=\int_MF^*\o.
\]
\end{proposition}

Suppose $T$ satisfies the conditions in Theorem \ref{res-m}. Let $W_0=W\ssm \tilde S$ and $W_1$ a \nbd\ of $\tilde S$ and 
consider the covering $\W=\{W_0,W_1\}$.
Let $\vP_T$ be represented by a \v CdR cocycle $(0,\psi_1,\psi_{01})$ in $A^p(\W,W_0)$.
For each $\l$ we take a regular \nbd\ $U_\l$ of $S_\l$ in $M$ \st\
$F(U_\l)\subset W_1$ and that $U_\l\cap U_\mu=\emptyset$ if $\l\ne\mu$. 
For each $\l$, we take a compact sub\mfd\ $R_\l$ of dimension $m$ with $C^\infty$ boundary in $U_\l$, containing $S_\l$ in its interior. 
From \eqref{corralexform} and \eqref{corralex}, we have\,:

\begin{proposition}\label{propint} {\rm (1)} The residue $\op{Res}(F^*\vP_T,\, S_\l)$ in $H_{m-p}(S_\l)$ is represented
by a cycle $C$ \st\
\[
\int_{R_\l}F^*\psi_1\wedge\varphi+\int_{R_{0\l}}F^*\psi_{01}\wedge\varphi=\int_C\varphi
\]
for any closed form $\varphi$ in $A^{m-p}(U_\l)$.
\vv

\noindent
{\rm (2)} In the \iso\ 
$H_{m-p}(S_\l)\simeq H^{m-p}_{\rm dR}(U_\l)^*$,  $\op{Res}(F^*\vP_T,\, S_\l)$ corresponds to the
 functional on $H^{m-p}_{\rm dR}(U_\l)$ that assigns to $[\varphi]$ the left hand side above.
\vv

\noindent
{\rm (3)} In particular, if $p=m$, the residue is a number given by
\[
\op{Res}(F^*\vP_T,\, S_\l)=\int_{R_\l}F^*\psi_1+\int_{R_{0\l}}F^*\psi_{01}.
\]
\end{proposition}

\begin{example}\label{rescycle}
Let $C$ be a finite $(m'-p)$-cycle on $W$, $\tilde S=|C|$ and $S=F^{-1}\tilde S$. We take $\vP_C$ as $\vP_{T_C}$. Then
\[
M\cdot_F T_C=M\cdot_F [C],\qquad \op{Res}(F^*\vP_C,\, S_\l)=(M\cdot_F C)_{S_\l}
\]
and the residue formula becomes
\[
M\cdot_F [C]=\sum_\l(i_\l)_*(M\cdot_F C)_{S_\l}\qquad\text{in}\ \ \breve H_{m-p}(M).
\]
In particular, if $M$ is compact and $p=m$,
\[
M\cdot_F [C]=\sum_\l(M\cdot_F C)_{S_\l}.
\]

Let $\o$ be a de~Rham representative of $C$. Then $T_\o$ is localized at $\tilde S$. As $\vP_{T_\o}$ we may take 
the class represented by a cocycle of the form $(0,\o,-\psi)$ (cf. Example \ref{smoothcase}). As a homology class,
$M\cdot_F T_\o=M\cdot_F [C]$. As a functional, it is given as in Proposition \ref{prodindR}. Also $\op{Res}(F^*\vP_{T_\o},\, S_\l)$
is a functional given as in Proposition \ref{propint} with $\psi_1=\o$ and $\psi_{01}=-\psi$.
\end{example}

See Propositions \ref{propisolated} and \ref{coinnonisol} below for  explicit expressions of $(M\cdot_F C)_{S_\l}$ in some special cases.

\begin{example}\label{exmult}
Let $W$ be a complex \mfd\ of dimension $n'$ and $M$ a complex sub\mfd\ of dimension $n$.
Also let $V$ be an analytic subvariety of $W$ of dimension $k$. Recall that there exists a subanalytic triangulation of $W$
compatible with $M$, $V$ and $\op{Sing}(V)$, the singular set of $V$. Thus $V$ may be thought of as a chain, which is not $C^\infty$
but still has the associated current $T_V$ of integration.
Moreover it is a cycle, as the real codimension of $\op{Sing}(V)$ in $V$ is greater than or equal to two. If $n+k=n'$ and if $p$ is an
isolated point of $M\cap V$, we have
\[
(M\cdot V)_p\ge \op{mult}_{p}(V),
\]
the multiplicity of $V$ at $p$. The equality holds, by definition, if $M$ is general with respect to $V$, i.e., the intersection of the
tangent space of $M$ at $p$
and the tangent cone of $V$ at $p$ consists only of $p$.
Note that $\op{mult}_{p}(V)$ coincides with the Lelong number of $T_V$ at $p$ (e.g., \cite[Ch.3, 2]{GH}).
\end{example}

We finish this section by giving a formula for the residue at  a non-isolated component. Thus,  in the  situation of 
Theorem \ref{res-m}, suppose that $S_\l$ is an oriented sub\mfd\ of $M$
of dimension $m-p$ with orientation compatible with that of $M$ in the sense described in Subsection \ref{Thom} (b). Let $p_\l$ be a point in $S_\l$ and $B_\l$ a small open ball of dimension $p$ in $M$ transverse
to $S_\l$ at $p_\l$. We orient $B_\l$ so that the orientation of $B_\l$ followed by that of $S_\l$ gives the
orientation of $M$. Setting $F_\l=F|_{B_\l}$, we have the commutative diagram
\[
\SelectTips{cm}{}
\xymatrix{ H^p(W,W\ssm \tilde S) \ar[d]_{F_\l^*}\ar[r]^-\sim_-{A_{W,\tilde S}}&H_{m'-p}(\tilde S) 
\ar[d]^{(B_\l\cdot_{F_\l}\ )_{p_\l}} \\
 H^p(B_\l,B_\l\ssm p_\l)\ar[r]^-\sim_-{A_{B_\l,p_\l}}& H_0 (p_\l).}
\]
We have the residue $\op{Res}(F_\l^*\vP_T,p_\l)=(B_\l\cdot_{F_\l} A(\vP_T))_{p_\l}$ in $H_0 (p_\l)\simeq\C$ so that it is a number.

\begin{theorem}\label{resnoniso} In the  situation of Theorem \ref{res-m}, suppose that $S_\l$ is an oriented  sub\mfd\ of $M$ of dimension $m-p$ and let $p_\l$  and $F_\l$ be as above. Then we have\,:
\[
\op{Res}(F^*\vP_T,S_\l)=\op{Res}(F_\l^*\vP_T,p_\l)\cdot S_\l\qquad\text{in}\ \ H_{m-p}(S_\l).
\]
\end{theorem}

\begin{proof} We try to find $\op{Res}(F^*\vP_T,S_\l)$ by  Proposition \ref{propint}. As $U_\l$, we take a tubular \nbd\ of $S_\l$ with a $C^\infty$ projection
$\pi:U_\l\ra S_\l$, which gives $U_\l$ the structure of a bundle of open balls of dimension $p$. Setting $U_0=U_\l\ssm S_\l$, we consider the covering $\U_\l=\{U_0,U_\l\}$ of  $U_\l$.
As $R_\l$, we take a bundle on $S_\l$ of closed balls of dimension $p$ in $U_\l$.
Then $R_{0\l}$ is a bundle on $S_\l$ of $(p-1)$-spheres.
We denote the restrictions of $\pi$ to $R_\l$ and $R_{0\l}$ by $\pi_\l$ and $\pi_{0\l}$, \r.
For a closed $(m-p)$-form $\varphi$  on $U_\l$, we compute the integral
\[
I:=\int_{R_\l}F^*\psi_1\wedge\varphi+\int_{R_{0\l}}F^*\psi_{01}\wedge\varphi.
\]

Since $\pi$ induces an \iso\ $\pi^*:H^{m-p}_{\rm dR}(S_\l)\stackrel\sim\ra H^{m-p}_{\rm dR}(U_\l)$, there
exist a closed $(m-p)$-form $\t$ on $S_\l$ and an $(m-p-1)$-form $\tau$ on $U_\l$ \st\
\[
\varphi=\pi^*\t+d\tau.
\]
Using the projection formula, the fact that $dF^*\psi_1=0$ and the Stokes formula, we have
\[
\int_{R_\l}F^*\psi_1\wedge\varphi=\int_{S_\l}(\pi_\l)_*F^*\psi_1\cdot\t+
(-1)^{p+1}\int_{R_{0\l}}F^*\psi_1\wedge\tau,
\]
where $(\pi_\l)_*$ denotes the integration along the fiber of $\pi_\l$. Note that $(\pi_\l)_*F^*\psi_1$ is a 
$C^\infty$ \fcn\ on $S_\l$. Noting that $dF^*\psi_{01}=F^*\psi_1$ on $U_{0\l}$ and $\partial R_{0\l}=\emptyset$, we also compute to get
\[
\int_{R_{0\l}}F^*\psi_{01}\wedge\varphi=\int_{S_\l}(\pi_{0\l})_*F^*\psi_{01}\cdot\t+
(-1)^p\int_{R_{0\l}}F^*\psi_1\wedge\tau.
\]
Thus we have
\[
I=\int_{S_\l}((\pi_\l)_*F^*\psi_1+(\pi_{0\l})_*F^*\psi_{01})\cdot\t
\]

Now recall that we have the integration along the fiber on the \v CdR cochains\,:
\[
\pi_*:A^q(\U_\l,U_0)\lra A^{q-p}(S_\l),
\]
which assigns to $\sigma=(0,\sigma_\l,\sigma_{0\l})$ the form $(\pi_\l)_*\sigma_\l+(\pi_{0\l})_*\sigma_{0\l}$
on $S_\l$. Moreover it is compatible with the differentials $D$ and $d$ (cf. \cite[Ch.II, 5]{Su6}). Since 
 $(0, F^*\psi_1|_{U_\l},F^*\psi_{01}|_{U_{0\l}})$ is a \v CdR cocycle
in $A^p(\U_\l,U_0)$, the \fcn\ $(\pi_\l)_*F^*\psi_1+(\pi_{0\l})_*F^*\psi_{01}$ is $d$-closed so that it is a 
constant. By definition the constant is exactly $\op{Res}(F_\l^*\vP_T,p_\l)$ above. Finally from
\[
\int_{S_\l}\t=\int_{S_\l}\varphi,
\]
we have the theorem.
\end{proof}

\begin{remark} 1. In the above situation,  $H_{m-p}(S_\l,\Z)\simeq\Z$ and is generated by the fundamental class $S_\l$. Thus if 
$\op{Res}(F_\l^*\vP_T,p_\l)$ is an integer, $\op{Res}(F^*\vP_T,S_\l)$ is an integral class.
\smallskip

\noindent
2. The above theorem  can also be proved topologically as \cite[Theorem 4.1.1, see also Theorem 7.3.2]{Su8}.
In fact, using  techniques and results in \cite{Su8}, we may compute residues in various settings. 
\end{remark}

\section{Coincidence point theorem}\label{seccoin}

\subsection{Coincidence homology classes and   indices}\label{local-coin}

Let $M$ and $N$ be connected and oriented  $C^\infty$ \mfd s of dimensions $m$ and $n$, \r, with $m\ge n$.
We set $W=M\times N$ and 
orient $W$ so that the orientation of $M$ followed by that of $N$ gives the orientation of $W$.
Let $f,\, g:M\ra N$ be $C^\infty$ maps
and denote by $\vG_f$ and $\vG_g$ the graphs of $f$ and  $g$  in $W$. We consider the map
\[
\tilde f:M\lra\vG_f\subset W\qquad\text{defined by}\ \ \tilde f(x)=(x,f(x)),
\]
which is a diffeomorphism onto $\vG_f$. We orient $\vG_f$ so that $\tilde f$  is orientation preserving.
Similarly we define $\tilde g$ for $g$. Recall the diagram (cf. Definition \ref{prodmap})\,:
\[
\SelectTips{cm}{}
\xymatrix{ H^n(W)  \ar[d]_{\tilde f^*}\ar[r]^-\sim_-{P}&\breve H_m(W) 
\ar[d]^{M\cdot_{\tilde f}\ } \\
 H^n(M)\ar[r]^-\sim_-{P}&  \breve H_{m-n} (M).}
\]

\begin{definition}\label{gcoin} The {\em global coincidence class $\vL(f,g)$  of the pair $(f,g)$} is defined by
\[
\vL(f,g)=M\cdot_{\tilde f}[ \vG_g]\qquad\text{in}\ \ \breve H_{m-n}(M).
\]
\end{definition}

Note that $\tilde f$ induces an \iso\ $\tilde f_*:\breve H_{m-n}(M)\stackrel\sim\ra \breve H_{m-n}(\vG_f)$
and   $\vL(f,g)$ corresponds to $\vG_f\cdot [ \vG_g]$ in $\breve H_{m-n}(\vG_f)$, which is sent to
$[\vG_f]\cdot [ \vG_g]$ in $\breve H_{m-n}(W)$ by the canonical \homo\ 
$\breve H_{m-n}(\vG_f)\ra\breve H_{m-n}(W)$.

We define the {\em coincidence point set} of the pair $(f,g)$ by
\[
\op{Coin}(f,g)=\{ \,p \in M\mid f(p)=g(p) \,\}.
\]
Note that $\op{Coin}(f,g)=\tilde f^{-1}(\vG_g)$. For shortness, we set $S=\op{Coin}(f,g)$. 

From now on we assume that $M$ is compact
so that $\vG_g$ and $S$ are compact. Recall the diagram (cf. Definition \ref{prodmap})\,:
\[
\SelectTips{cm}{}
\xymatrix{ H^n(W,W\ssm \vG_g) \ar[d]_{\tilde f^*}\ar[r]^-\sim_-{A}&  H_m(\vG_g) 
\ar[d]^{(M\cdot_{\tilde f}\ )_S} \\
H^n(M,M\ssm S)\ar[r]^-\sim_-{A}&  H_{m-n} (S).}
\]

\begin{definition}\label{lcoin}
The {\em local coincidence class $\vL(f,g;S)$ of the pair $(f,g)$ at $S$} is defined to be the localized intersection class\,:
\[
\vL (f,g;S)=(M\cdot_{\tilde f} \vG_g)_S\qquad\text{in}\ \  H_{m-n}(S).
\]
\end{definition}

Note that $\tilde f$ induces a \homo\ $\tilde f_*:H_{m-n}(S)\ra H_{m-n}(\vG_f)$
and   $\vL(f,g;S)$  is sent to
$\vG_f\cdot [ \vG_g]$.

\begin{remark} 1. The classes $\vL(f,g)$ and $\vL(f,g;S)$ are in fact in homology with $\Z$ coefficients.
\smallskip

\noindent
2. We have
\[
\vL(g,f)=(-1)^n\vL(f,g),\qquad \vL(g,f;S)=(-1)^n\vL(f,g;S).
\]
\end{remark}

Suppose $S=\op{Coin}(f,g)$ has a finite number of connected components $(S_\l)_\l$.
Then we have $H_{m-n}(S)=\bigoplus_\l H_{m-n}(S_\l)$ and accordingly we have the
local coincidence class $\vL(f,g; S_\l)$ in $H_{m-n}(S_\l)$. From Theorem \ref{res-m}, we have
a general  coincidence point theorem\,:

\begin{theorem}\label{gcth} In the above situation
\[
\vL(f,g)=\sum_\l (i_\l)_*\vL(f,g; S_\l)\qquad\text{in}\ \  H_{m-n}(M).
\]
\end{theorem}

In general, $\vL(f,g)$ and $\vL(f,g; S_\l)$ are given as in Propositions \ref{prodindR} and \ref{propint}.
The theorem becomes more meaningful if we have explicit descriptions of them.

In the case  $m=n$, $\vL(f,g; S_\l)$ is in $H_0(S_\l)=\C$ so that it is a number
(in fact an integer), which we call the {\em coincidence index} of $(f,g)$ at $S_\l$. If $S_\l$ consists of a point $p$, we have the following explicit formula.
In fact it is already known, however we give an alternative short  proof using the  Thom class in the \v Cech-de~Rham cohomology.
Let $U$ be a coordinate \nbd\ around $p$ with  
coordinates $x=(x_1,\dots,x_m)$ in $M$ and $V$ a coordinate \nbd\ around $f(p)=g(p)$ with  
coordinates $y=(y_1,\dots,y_m)$ in $N$. Also let $D$ be a closed ball around $p$ in $U$ \st\ $f(D)\subset V$ and $g(D)\subset V$.
Thus we may consider the map $g-f:D\ra\R^m$ whose image is the origin $0$ in $\R^m$ only at $p$. The boundary $\partial D$
is homeomorphic to the unit sphere $S^{m-1}$ and we have the map
\[
\gamma:\partial D\lra S^{m-1}\qquad\text{defined by}\ \ \gamma(x)=\frac{g(x)-f(x)}{\Vert g(x)-f(x)\Vert}.
\]
We denote the degree of this map by $\op{deg}(g-f,p)$.

\begin{proposition}\label{propisolated} In the above situation
\[
\vL(f,g; p)=\op{deg}(g-f,p).
\]
\end{proposition}

\begin{proof} Let $\vP_{\vG_g}$ be the Thom class of $\vG_g$ and $(0,\psi_1,\psi_{01})$ its \v CdR representative.
We may take $D$ as $R_\l$ in Proposition \ref{propint}. Since $R_{0\l}=-\partial D$, we have
\begin{equationth}\label{coiniso}
\vL(f,g; p)=\int_D\tilde f^*\psi_1-\int_{\partial D}\tilde f^*\psi_{01}.
\end{equationth}
Recall that $\vP_{\vG_g}$ may be naturally identified with the Thom class of  the normal bundle $N_{\vG_g}$,
which is trivial over $\tilde g(U)$\,;
$N_{\vG_g}|_{\tilde g(U)}\simeq \R^m\times \tilde g(U)$ (cf. Subsection \ref{Thom} (b)). Let 
$\rho_g:N_{\vG_g}|_{\tilde g(U)}\ra \R^m$ denote the projection onto the fiber direction.  Also let $\psi_m$ be 
an angular form on $\R^m\ssm\{0\}$. Then on $\tilde g(U)$ the Thom class of $N_{\vG_g}$ is represented by the cocycle (cf. Proposition \ref{Thomloc})
\[
(0,0,-\rho_g^*\psi_m).
\]

Let $\pi_1$ and $\pi_2$ denote the projections of $W=M\times N$ onto the first and second factors, respectively. We set 
$x=\pi_1^*x$ and $y=\pi_2^*y$ on $U\times V$. By our orientation convention, in a \nbd\ of $\tilde g(p)$ in $N_{\vG_g}$, we may take $g(x)-y$ as fiber coordinates and $x$ as base coordinates of the bundle $N_{\vG_g}$ so that we may write 
$\rho_g(x,y)=g(x)-y$. Then \eqref{coiniso} becomes
\[
\vL(f,g; p)=\int_{\partial D}(\rho_g\circ\tilde f)^*\psi_m=\int_{\partial D}(g-f)^*\psi_m,
\]
which is nothing but $\op{deg}(g-f,p)$. 
\end{proof}

In the above situation, let $J_f(p)$ and 
$J_g(p)$ denote the Jacobian matrices of $f$ and $g$ at $p$. A coincidence point $p$ of the pair $(f,g)$ is said to be {\em non-degenerate} if
\[
\det(J_g(p)-J_f(p))\ne 0.
\]

\begin{corollary} If $p$ is a non-degenerate coincidence point,
\[
\vL(f,g; p)=\op{sgn}\, \det(J_g(p)-J_f(p)).
\]
\end{corollary}

Now consider the case $m>n$. Suppose $S_\l$ is an oriented sub\mfd\ of $M$ of dimension $m-n$.  Let $p_\l$ be a point in $S_\l$ and $B_\l$ a small open ball of dimension $n$ in $M$ transverse
to $S_\l$ at $p_\l$. Setting $f_\l=f|_{B_\l}$ and $g_\l=g|_{B_\l}$, we have $\deg(g_\l-f_\l,p_\l)$.
From Theorem \ref{resnoniso}, we have\,:

\begin{proposition}\label{coinnonisol} Let $S_\l$ be a connected component of $\op{Coin}(f,g)$.
If $S_\l$ is an oriented sub\mfd\ of $M$ of dimension $m-n$,
\[
\vL(f,g;S_\l)=\op{deg}(g_\l-f_\l,p_\l)\cdot S_\l\qquad\text{in} \ \  H_{m-n}(S_\l).
\]
\end{proposition}

\subsection{Lefschetz coincidence point formula}

Let $M$ and $N$ be compact, connected and oriented $C^\infty$ manifolds of the same dimension $m$ and
let $f, g \colon M \ra N$ be $C^\infty$ maps. In this situation, $\breve H_0(M)=H_0(M)=\C$
and $\vL(f,g)$ is a number (in fact an integer), which has an explicit description.
Let
\[
H^p(f):H^p(N)\lra H^p(M)
\]
be the \homo\ induced by $f$ on the $p$-th  cohomology group and similarly for $H^p(g)$. We set $q=m-p$. The 
Poincar\'e duality  allows us to define the composition
\[
\begin{CD}
H^q(M)\simeq H^p(M)^* @>{H^p(g)^*}>> H^p(N)^* \simeq H^q(N) @>{H^q(f)}>> H^q(M).
\end{CD}
\]
We define the {\em Lefschetz coincidence number $L(f,g)$ of the pair $(f,g)$} as 
\[
L(f,g):=\sum_{q=0}^m (-1)^q \cdot \op{tr} (H^q(f) \circ H^{m-q}(g)^*).
\]

Although the following is already known, we include a proof for the sake of completeness. It is  a modification of the presentation as given in \cite{GH}
for the fixed point case, i.e., the case $M=N$ and $g=1_M$, the identity map of $M$.

\begin{proposition}\label{globallef}
In the above situation we have\,:
\[
\vL(f,g)=L(f,g).
\]
\end{proposition}

\begin{proof}
Let $\{\mu_{i}^{p} \}_i$ be a set of closed forms representing a basis of $H_{\rm dR}^p(M)$.
We set $q=m-p$ and let $\{\check{\mu}^{q}_{j} \}_j$ be a set of forms  representing a basis of $H_{\rm dR}^q(M)$ dual 
 to $\{[\mu_{i}^{p}]\}_i$\,:
\[
\int_M\mu_{i}^{p}\wedge \check{\mu}_{j}^{q}=\delta_{ij}.
\]
We also take a set of forms  $\{\nu_k^{p} \}_k$ representing a basis of $H_{\rm dR}^p(N)$
and  a set of forms $\{\check{\nu}^{q}_\ell \}_\ell$ representing a  basis of $H_{\rm dR}^q(N)$ dual  to  
$\{[\nu_k^{p}] \}_k$.
By the K\"unneth formula, a basis of $H_{\rm dR}^m(W)$, $W=M\times N$,  is represented by
\[
\left\{\,  \xi^{p,q}_{i,\ell}=\pi_1^* \mu_{i}^{p}\wedge \pi_2^* \check\nu_\ell^{q}\, \right\}_{p+q=m},
\]
where $\pi_1$ and $\pi_2$ are projections onto the first and second factors.

Note that in general, for a $p$-form $\o$ on $M$ and a $q$-form $\t$ on $N$, we have
\begin{equationth}\label{intprod}
\int_{\vG_f}\pi_1^*\o\wedge\pi_2^*\t=\int_M\o\wedge f^*\t
\end{equationth}
and similarly for the integration on $\vG_g$.

Let $G^p=(g^p_{ki})$ be the matrix representing $H^p(g)$ in the  bases  
$\{[\nu_k^{p}] \}_k$ and $\{[\mu_{i}^{p}] \}_i$\,:
\[
H^p(g)[\nu_k^p]=\sum_i g^p_{ik}[\mu_i^p].
\]
Thus the dual map $H^p(g)^*$  is represented by the transposed ${}^t\hskip-.2mm G^p$ in the  bases  
$\{[\check\mu_j^q] \}_j$ and $\{[\check\nu_\ell^{q}] \}_\ell$.
Also let $\check F^q=(\check f^q_{\ell j})$ be the matrix representing 
$H^q(f)$  in the  bases  $\{[\check\nu_\ell^{q}] \}_\ell$ and $\{[\check\mu_j^q \}_j$\,:
\[
H^q(f)[\check\nu_\ell^{q}]=\sum_j\check f^q_{j\ell}[\check\mu_j^q].
\]

Let $\eta_g$ be an $m$-form representing the Poincar\'e dual of $[\vG_g]$ in $W$.
Using \eqref{corrpoin} and \eqref{intprod} for $\vG_g$, we compute to get
\[
[\eta_g]=\sum_{q,i,\ell}(-1)^q g^{p}_{i\ell}\, [\xi^{p,q}_{i,\ell}].
\]
Thus we have
\[
\vL(f,g)=\int_M\tilde f^*\eta_g
=\sum_{q,i,\ell}(-1)^{q}g^{p}_{i\ell}\int_M \mu_{i}^{p}\wedge f^*\check\nu_\ell^{q}
=\sum_{q,i,\ell}(-1)^{q}g^{p}_{i\ell}\,  \check f^q_{i\ell}.
\]
Since $ \check f^q_{i\ell}$ is the $\ell i$ entry of $\check F^q$ and $g^{p}_{i\ell}$ the $i\ell$ entry of 
${}^t\hskip-.2mm G^{p}$, we have the proposition.
\end{proof}

From Theorem \ref{gcth} and  Propositions \ref{propisolated} and \ref{globallef}, we have\,:

\begin{theorem}\label{non-isolatecoin} Let $M$ and $N$ be compact oriented $C^{\infty}$ manifolds of same dimension
and let $f , g \colon M \ra N$ be $C^\infty$ maps.
Suppose $\op{Coin}(f,g)$ has a finite number of connected components $(S_\l)_\l$. Then 
\[
L(f,g)=\sum_\l \vL(f,g;S_\l).
\]
\end{theorem}

In  the case  the set of coincidence points consists only of isolated points, we have\,:
\begin{corollary}[Lefschetz coincidence point formula] Let $M$ and $N$ be compact oriented $C^{\infty}$ 
manifolds of the same dimension
and let $f , g \colon M \ra N$ be $C^\infty$ maps.
Suppose $\op{Coin}(f,g)$ consists   of a finite number of isolated points. Then
\[
L(f,g)=\sum_{p \in \op{Coin}(f,g) }\deg(g-f,p).
\]
Moreover, if all coincidence points are isolated and non-degenerate then 
\[
L(f,g)=\sum_{p\in \op{Coin}(f,g) }\op{sgn}\, \det(J_g(p)-J_f(p)).
\]\label{lefcpf}
\end{corollary}

\begin{remark} 1. The above theory applied to the case $N=M$ and $g=1_M$, the identity map of $M$, gives a general
fixed point theorem for $f$ and the Lefschetz fixed point formula, which is  effective also in the study of periodic
points.
\smallskip

\noindent
2. Let  $f,g:M\ra N$ be  $C^\infty$ maps.  If 
$g$ is a diffeomorphism, the coincidence theory for the pair $(f,g)$  is equivalent to
the fixed point theory for the map $g^{-1}\circ f$ of $M$.
\smallskip

\noindent
3. As a textbook dealing with the Lefschetz coincidence theory, we list  \cite{V}.
\smallskip

\noindent
4. In \cite{S}, for a pair of maps  from a topological space to a \mfd, the Lefschetz \homo\
is defined and it is proved that if it is non-trivial, then there is a coincidence point.
\end{remark}

\bigskip

C. Bisi

Dipartimento di Matematica ed Informatica

Universit\`a di Ferrara

Via Machiavelli, n. 35

44121 Ferrara, Italy.

bsicnz@unife.it

\bigskip

F. Bracci

Dipartimento di Matematica

Universit\`a di Roma ``Tor Vergata''

Via della Ricerca Scientifica, n. 1

00133 Roma, Italy.

fbracci@mat.uniroma2.it

\bigskip

T. Izawa

Department of Information and Computer Science

Hokkaido University of Science

Sapporo 006-8585, Japan

t-izawa@hus.ac.jp

\bigskip

T. Suwa

Department of Mathematics

Hokkaido University

Sapporo 060-0810, Japan

tsuwa@sci.hokudai.ac.jp


\begin{thebibliography}{9}


\bibitem{Bo1}  R. Bott, {\em  Lectures on characteristic classes and foliations},
Lectures on Algebraic and Differential Topology, Lecture Notes in
Math.  {\bf 279}, Springer-Verlag  1972, 1-94

\bibitem{BT}  R. Bott and L. Tu,
{\em Differential Forms in Algebraic Topology},   Graduate Texts in Math. {\bf 82},
Springer-Verlag 1982.

\bibitem{Br1}  J.-P. Brasselet,  {\em D\'efinition combinatoire des
homomorphismes de Poincar\'e, Alexander et Thom pour une
pseudo-vari\'et\'e},  Ast\'erisque {\bf 82-83}, Soc. Math.  France 1981, 71-91.

\bibitem{Fu} W. Fulton, {\em Intersection Theory},  Springer-Verlag 1984.

\bibitem{GH}  P. Griffiths and J. Harris,
{\em Principles of Algebraic Geometry}, John Wiley $\&$ Sons 1978.

\bibitem{Lef1}  S. Lefschetz, {\em Intersections and transformations of complexes and manifolds},
Trans. Amer. Math. Soc. {\bf 28} (1926), 1-49.

\bibitem{Lef2}  S. Lefschetz, {\em Manifolds with a boundary and their transformations},
Trans. Amer. Math. Soc. {\bf 29} (1927), 429-462.

\bibitem{Leh2} D. Lehmann,
{\em Syst\`emes d'alv\'eoles et int\'egration sur le complexe de \v Cech-de~Rham},
Publications de l'IRMA, 23, N$^o$ VI, Universit\'e de Lille I, 1991.

\bibitem{MS} J. Milnor  and J. Stasheff,  {\em Characteristic Classes},
Ann. of Math. Studies {\bf 76}, Princeton Univ. Press 1974.


\bibitem{S} P. Saveliev, {\em Lefschetz coincidence theory for maps between spaces of different dimensions}, Topology and its Applications {\bf 116}  (2001), 137-152.

\bibitem{Sh} M. Shiota, {\em Geometry of Subanalytic and Semialgebraic Sets}, Progress in Math., Birkh\"auser 1997.

\bibitem{Su6}  T. Suwa,
{\em Indices of Vector Fields and Residues of Singular Holomorphic Foliations},
Hermann 1998.

\bibitem{Su8} T. Suwa, {\em Residue Theoretical Approach to Intersection Theory},
Real and Complex Singularities, Contemporary Math. Amer. Math. Soc. {\bf 459} (2008), 207-261.

\bibitem{V} J.W. Vick, {\em Homology Theory: An Introduction to Algebraic Topology}, Graduate Texts in Math., Springer-Verlag 1994.
\end{thebibliography}
\end{document}